\documentclass[12pt]{article}
\usepackage{enumerate,amsfonts, graphicx, amsmath}
\usepackage{graphicx}
\usepackage{epsfig}
\usepackage{epstopdf}
\usepackage{multirow,color}
\usepackage{url}

\title{Detecting and correcting the loss of independence in
nonlinear conjugate gradient\thanks{Supported
in part by a grant from the U.~S.~Air Force Office of Scientific Research and in part
by a Discovery Grant from the Natural Sciences and Engineering Research Council
(NSERC) of Canada.}}
\author{Sahar Karimi\thanks{Department of Combinatorics \& Optimization,
University of Waterloo, 200 University Ave.~W., Waterloo, ON, N2L 3G1,
Canada, {\tt s2karimi@uwaterloo.ca}.} \and 
Stephen Vavasis\thanks{Department of Combinatorics \& Optimization,
University of Waterloo, 200 University Ave.~W., Waterloo, ON, N2L 3G1,
Canada, {\tt vavasis@uwaterloo.ca.}}}
\newtheorem{alg}{Algorithm}
\renewcommand\a{{\bf a}}
\renewcommand\b{{\bf b}}
\renewcommand\c{{\bf c}}
\renewcommand\d{{\bf d}}
\newcommand\g{{\bf g}}
\newcommand\q{{\bf q}}
\renewcommand\r{{\bf r}}
\newcommand\R{{\bf R}}
\newcommand\s{{\bf s}}
\newcommand\x{{\bf x}}
\newcommand\y{{\bf y}}
\newcommand\lt{\left}
\newcommand\rt{\right}
\newcommand\bdelta{\mbox{\boldmath{$\delta$}}}
\newtheorem{lemma}{Lemma}
\newtheorem{theorem}{Theorem}
\newcommand{\eref}[1]{$(\ref{#1})$}
\newenvironment{proof}{\noindent {\bf Proof. }}{\hfill $\square$}
\begin{document}
\maketitle
\begin{abstract}
It is well known that search directions in nonlinear conjugate gradient (CG) can
sometimes become nearly dependent, causing a dramatic slow-down in
the convergence rate.  
We provide a theoretical analysis of this loss of
independence.  The analysis applies to the case of a strictly convex
objective function and is motivated by older work of Nemirovsky and Yudin. Loss of
independence can affect several of the well-known variants of
nonlinear CG including Fletcher-Reeves, Polak-Ribi\`ere (nonnegative
variant), and Hager-Zhang.

Based on our analysis, we propose a relatively inexpensive
computational test for detecting loss of independence.  We also
propose a method for correcting it when it is detected, which we call
``subspace optimization.''  Although the correction method is somewhat
expensive, our experiments show that in some cases, usually the most
ill-conditioned ones, it yields a method
much faster than any of these three variants. 
Even though our theory covers only strongly
convex objective functions, we provide computational results to
indicate that the detection and correction mechanisms may also hold
promise for nonconvex optimization.
\end{abstract}


\section{Conjugate gradient}
The method of conjugate gradients  (CG)
was introduced by Hestenes and Stiefel \cite{hestenesstiefel}
for 
minimizing convex quadratic functions.  
We refer to this
algorithm as ``linear conjugate gradient.''
It was soon generalized
by Fletcher and Reeves \cite{fletcherreevescg} 
and Polak and Ribi\`ere \cite{polakribiere},
to the general problem of unconstrained minimization, i.e., 
\begin{equation}\label{probdef}
\min_{\x \in \mathbb R^n} f\left (\x \right).
\end{equation}
However, the theoretical basis for
nonlinear CG is considerably weaker than that of the linear
case.  In the linear case, the successive gradients are mutually orthogonal
and the search directions are mutually conjugate; these facts allow several
strong convergence proofs including finite termination, convergence bounded
in terms of problem condition number, and superlinear convergence \cite{golubvanloanbook}.
Indeed, the only thing that can go awry is loss of orthogonality due to roundoff
error.  Roundoff error can indeed be a significant problem in practice but is
not the main topic of our current study (although see Section~\ref{sec:divdiff}).

In the case of nonlinear conjugate gradient, there is no 
orthogonality of the search directions, and in fact, the
directions can become nearly dependent.  It is generally accepted in the
optimization community (and confirmed by our own experiments described
in Section~\ref{sec:comput}) that the Polak-Ribi\`ere variant is more
robust against dependent search directions than the Fletcher-Reeves
variant; see Nocedal and Wright \cite{nocedalwrightbook} for a discussion of this
issue.

The standard technique to combat loss of independence is restarting
the method, i.e., occasionally taking a step of pure steepest descent.
However, there is little rigorous theory that explains when to restart
the method.  The best known rigorous result in this direction is a
proof that when an iterate is sufficiently close to the root, if one
restarts every $n$ iterations, one is guaranteed $n$-step quadratic
convergence to the optimizer.  Here, $n$ denotes the number of
variables.  This result is unsatisfying for at least two reasons.  First,
there is no apparent method to detect when an iterate is sufficiently
close to the root in order to apply this theorem.  Second, restarting
every $n$ iterations does not seem to be practically motivated.  The
reason is that the convergence of conjugate gradient, both linear
and nonlinear, is much more
closely tied to the conditioning of the problem than to $n$, the
number of variables.  Thus, one would apparently prefer a rigorously
supported restart strategy that is condition-dependent rather problem
size-dependent.

In this paper we turn in Section~\ref{sec:NYanalysis}
to a decades-old analysis of a variant of
conjugate gradient by Nemirovsky and Yudin 
\cite{nemiryudinbook} that is intended
for the case of a strongly convex objective function.  
We will argue
in Section~\ref{sec:detection} that
the analysis of their algorithm suggests a rigorous
way, at least for this class of objective functions, to detect loss of
independence in the search directions.  Armed with this knowledge, we
are then able to propose a method for correcting loss of independence,
which is described in Section~\ref{sec:correction}.

The detection procedure is relatively cheap; the correction procedure,
however, is quite expensive.  Nonetheless, nonlinear conjugate
gradient (any variant) augmented by our correction procedure in
practice is sometimes the fastest method for solving
the problem, according to our experiments detailed in Section~\ref{sec:comput}.
Furthermore, if the correction procedure is used, then one obtains a
theoretical bound on the number of iterations that is the same as
Nemirovsky and Yudin's and is the best possible convergence bound
known to date (although we do not achieve their bound on function/gradient
evaluations; see further remarks below).  
In contrast, there is no comparable convergence bound known 
for any of the standard CG methods.
Indeed, Nemirovsky and Yudin
argue that their worst-case complexity for strongly convex
functions is quite poor.
A strength of our proposed correction method is that it requires no prior knowledge
of parameters of the underlying function, unlike most methods that
achieve the theoretical convergence bound.

Methods reviewed in this paper are among techniques that are generally
referred to as ``first-order algorithms'' because they use only the
first derivative information of the function in each iteration. Due to
the successful theory developed first by
Nemirovsky and Yudin and extended by Nesterov, first-order
algorithms have attracted many researchers during the last decade and
have been extended to solving different classes of problems. Nesterov
in \cite{nesterov2005} proposed a variation of his earlier algorithms
for minimizing a nonsmooth function. In addition to nonsmooth
optimization, Nesterov's algorithm has been adapted for constrained
problems with simple enough feasible regions so that a projection on
these sets can be easily computed. One may refer to \cite{tseng2008}
and references therein for a more in-depth discussion of different
adaptations of Nesterov's algorithm. The focus of this paper, however,
is more on the CG algorithm and not on first-order techniques in general.

Hestenes and Stiefel's original linear CG has the following form:
\begin{subequations}
\begin{equation}\label{xiteratel}
\x^{j+1}=\x^j - \frac{\left(\r^j\right)^t\d^j}{\left(\d^j\right)^tA\d^j}\d_j,
\end{equation}
\begin{equation}\label{diteratel}
\d^{j+1}=-\r^{j+1}+\frac{\left(\r_{j+1}\right)^tA\d^j}{\left(\d^j\right)^tA\d_j}\d_j.
\end{equation}
\end{subequations}
In the above equations $\r^j$ is $\nabla f(\x)=A\x^j-\b$ and
$\d^0=-\r^0$. It is possible to show that the number of iterations in linear CG is bounded by the dimension of the problem, $n$. For more
details on linear CG, one may refer to \cite{golubvanloanbook} or
\cite{nocedalwrightbook}.

Nonlinear CG was proposed by Fletcher and
Reeves \cite{fletcherreevescg} as an adaptation of the above algorithm
for minimizing a general nonlinear function. The general form of this
algorithm is as follows:
\begin{subequations}
\begin{equation}\label{xiteratenl}
\x^{j+1}=\x^j+\alpha^j \d^j,
\end{equation}
\begin{equation}\label{diteratenl}
\d^{j+1}=-\g^{j+1}+\beta^{j}\d^j.
\end{equation}
\end{subequations}

Here, $\d^j$ is the {\em search direction} at each iteration, $\g^{j+1}$
is the gradient of the function at $(j+1)$th iterate, i.e., $\nabla
f(\x^{j+1})$; and $\alpha^j$ is the step size, usually determined by a
line search. Different updating rules for $\beta^j$ give us different
variants of nonlinear CG. The most common formulas for computing $\beta^j$
are:

\begin{equation*}
\begin{array}{ll}
\text{Fletcher-Reeves (1964):}&  \beta_{FR}=\frac{\| \g^{j+1}\|}{\| \g^{j}\|}, \\
\text{Polak-Ribi\`ere (1969):} & \beta_{PR}=\frac{\left( \g^{j+1}\right)^t\left(\g^{j+1}-\g^j\right)}{\| \g^{j}\|}.
\end{array}
\end{equation*}
Hager and Zhang \cite{zhangsurvey} present a complete list of all
updating rules in their survey on nonlinear CG. The convergence of nonlinear CG is
highly dependent on the line search; for some, the exact line search is
crucial. There are numerous papers devoted to the study of global
convergence of nonlinear CG algorithms, most of which discuss variants of nonlinear CG
that do not rely on exact line search to be globally
convergent. Al-Baali \cite{albaalipaper} shows the convergence of
Fletcher-Reeves algorithm with inexact line search. Gilbert and Nocedal
\cite{gilbertnocedalpaper} establish the convergence of 
a variant of the Polak-Ribi\'ere 
nonlinear CG algorithm
with no restart and no exact line search. Dai and Yuan
\cite{daiyuanpaper} present a nonlinear CG for which the standard Wolfe
condition suffices. A recent variant of CG has been proposed by Hager
and Zhang \cite{hagerzhangdescentcg} that relies on a line search
satisfying the Wolfe Conditions. Furthermore this algorithm has the
advantage that every search direction is a descent direction, which is
not necessarily the case in nonlinear CG.

{From} Yuan and Stoer's perspective \cite{yuanstoer}, CG is a technique
in which the search direction $\d^{j+1}$ lies in the subspace spanned by 
$\text{\emph {Sp}}\{\g^{j+1}, \d^j\}$. In the algorithm they propose they compute
the new search direction by minimizing a quadratic approximation of
the objective function over the mentioned subspace. A more generalized
form of CG called Heavy Ball Method, was introduced by Polyak
\cite{polyakbook}, in which $\x^{j+1}$ is $\x^j+\alpha (-\g^{j}) + \beta
(\x^j-\x^{j-1})$. He proved a geometric progression rate for this
algorithm when $\alpha$ and $\beta$ belong to a specific range.

\section{An analysis of the loss of independence}
\label{sec:NYanalysis}
The analysis in this section focuses on strongly convex objective
functions.   We say that
$f$ is {\em strongly convex with parameters $(L,l)$} if
for any $\x,\y$ lying in the level set of $\x^0$,
\begin{eqnarray}
\| \nabla f(\x) -\nabla f(\y)\| &\le & L \| \x-\y \|, 
\label{introeq1}
\\
f(\y) -f(\x) &\ge& \left\langle  \nabla f(\x), \y-\x \right\rangle + \frac{l}{2} \| \y-\x\|^2.
\label{introeq2}
\end{eqnarray}
For example, in the case of a convex quadratic function
$f(\x)=\x^tA\x/2-\b^t\x$ where $A\in\R^{n\times n}$ is
symmetric positive definite, $L/l$ is the condition number of $A$.
{From} inequality \eqref{introeq1}, it follows that
\begin{equation}\label{introeq3}
f(\y)-f(\x) \le\left\langle  \nabla f(\x), \y-\x \right\rangle + \frac{L}{2} \| \y-\x\|^2,
\end{equation} 
which will be useful in our analysis.
We follow the standard notation throughout
this paper: $\langle \cdot , \cdot \rangle$ represents the inner product of
two vectors in proper dimension, and $\| \cdot \|$ stands for the 2-norm
of a vector unless otherwise is stated. Bold lower case characters and
upper case characters are used for vectors and matrices respectively;
and their superscript states the iteration count.

In \cite{nemiryudinbook}, Nemirovsky and Yudin propose
an algorithm for minimizing a strongly convex $f$ that
achieves a worst-case complexity bound of
$O(\ln(1/\epsilon)\sqrt{L/l})$. Here $\epsilon$ is the desired relative
accuracy, that is,
$\epsilon=(f(\x^n)-f(\x^*)) / (f(\x^0)-f(\x^*))$, where
$\x^0$ is the starting point, $\x^*$ is the optimizer, and
$\x^n$ is the final iterate.  This bound is still the best known for
this particular class of methods and functions.  Their algorithm can
be regarded as a variant of conjugate gradient.  

The NY algorithm has never been widely used in practice for several
reasons.  First, when applied to convex quadratic functions, it does
not reduce to linear conjugate gradient and in fact can be much
slower.  (In contrast, the FR and PR variants of nonlinear CG reduce
to linear CG in the case of a convex quadratic and if an exact line
search is used.  Many would argue that this is a defining property of
nonlinear conjugate gradient.)  Second, the method requires an
expensive subspace optimization step on every iteration.  Our
correction procedure involves a related subspace optimization; we
comment on its cost in Section~\ref{sec:correction}.  A later paper
by Nesterov \cite{nesterov1983} remedied this
drawback by achieving the same complexity
without the need for subspace optimization.  Third, the Nemirovsky-Yudin
algorithm requires prior knowledge of $L/l$, which may not be available
in practice.  Furthermore, for some classes of problems, e.g.,
log-barrier functions, the upper bound on $L/l$ varies wildly
depending on the choice of starting point.  Therefore, we would much
prefer methods that do not require prior knowledge of such 
parameters.\footnote{Subsequent to the release
of an earlier draft of the present manuscript,
Nesterov \cite{nesterov2013} also considered the issue of
optimal methods that do not require prior knowledge of parameters.}

For convenience, let us represent the gradient at $\x^j$, that is,
$g(\x^j)$, by $\g^j$, and let $v_f(\x)$ denote the residual of the function, i.e. $f(\x)-f(\x^*)$.
Our main lemma requires the following three properties:

\begin{enumerate}[(a)]
\item $f(\x^{j+1})\le f(\x^j)-\frac{1}{2L}\| \g^j\|^2$\label{prop1}
\item $\left\langle \g^j, \x^*-\x^j\right\rangle \le f(\x^*)- f(\x^j)$\label{prop2}
\item $v_f(\x^0)=f(\x^0)-f^*\ge \frac{l}{2} \Vert \x^*-\x^0\Vert^2$\label{prop3}
\end{enumerate}

Property \eqref{prop1} assumes that the step computed by the algorithm
is at least as good as steepest descent with fixed step length of $1/L$.
Property \eqref{prop2} is true by convexity of the function, 
and property \eqref{prop3} is a direct derivation from inequality  \eqref{introeq2}. 
We are now ready to present the main lemma that yields
a complexity bound for conjugate gradient.

 \begin{lemma}\label{lem1} Consider applying nonlinear conjugate gradient
(any variant) to strongly convex function $f(\x)$.  Assume that
the step at each iteration satisfies \eref{prop1}.
Furthermore, 
suppose $m\ge \lt\lceil8\rho\sqrt{\frac{L}{l}}\rt\rceil$ and 
 \begin{equation}\label{eq1} \frac{\left(f(\x^{m-1})-f(\x^0)\right)}{4} 
\left(\sum_{j=0}^{m-1} \lambda^j\right) + 
\sum_{j=0}^{m-1} \lambda^j \left\langle \g^j,\x^j-\x^0\right\rangle <0,
\end{equation}  
 and 
 \begin{equation}\label{eq2} \left\| \sum_{j=0}^{m-1} \lambda^j\g^j\right\| 
\le \rho \sqrt{ \sum_{j=0}^{m-1} \left(\lambda^j\right)^2 \left\| \g^j\right\|^2}, 
\end{equation}
are satisfied, where $\rho$ is a constant $\ge 1$, and $$\lambda^j=\sqrt{\frac{f(\x^j)-f(\x^{j+1})}{\| \g^j\|^2}}.$$ Then the residual of the function is divided in half after $m$ iterations; i.e. $v_f(\x^m)\le \frac{1}{2} v_f(\x^0)$.
\end{lemma}

\noindent
{\bf Remark.}  For the remainder of this paper, we regard
conditions \eref{eq1} and \eref{eq2} stated in the above lemma as quantification
of the independence of succesive search directions.  
In other words, we define the phrase ``loss of independence'' to mean
failure of these inequalities.
For example,
in the case of linear conjugate gradient, \eref{eq1} is automatically
satisfied because 
$\left\langle \g^j,\x^j-\x^0\right\rangle=0$ by orthogonality of 
gradients.  In addition, \eref{eq2} is satisfied as an equality with $\rho=1$
by linear conjugate gradient because in this case it reduces to Pythagoras's equation.
Thus, loss of independence never occurs in linear conjugate (in exact arithmetic).

\begin{proof}
Our proof is an extension of the proof in section 7.3 in \cite{nemiryudinbook}. Suppose by contradiction that $m\ge \left \lceil8\rho\sqrt{\frac{L}{l}}\right \rceil$, \eqref{eq1} and \eqref{eq2} are satisfied; but $v_f(\x^m)>\frac{v_f(\x^0)}{2}$.

By definition of $\lambda^j$, 
$$f(\x^{j+1}) = f(\x^j)- \left(\lambda^j\right)^2 \| \g^j\|^2,$$
hence 
$$v_f(\x^{j+1})= v_f(\x^j)-\left(\lambda^j\right)^2\| \g^j\|^2.$$

Summing these inequalities over $j=0,\ldots,m-1$, we get:
$$0\le v_f(\x^m) = v_f(\x^0)-\sum_{j=0}^{m-1} \left(\lambda^j\right)^2\| \g^j\|^2,$$
or equivalently,  
\begin{equation}\label{eq3} \sum_{j=0}^{m-1}\left(\lambda^j\right)^2\| \g^j\|^2\le v_f(\x^0). \end{equation}

By convexity of the function we have,
$$\left\langle  \g^j, \x^*-\x^j\right\rangle \le f(\x^*)-f(\x^j)=-v_f(\x^j),$$ and so
$$\left\langle  \g^j, \x^*-\x^0\right\rangle - \left\langle  \g^j, \x^j-\x^0\right\rangle \le -v_f(\x^j)\le -v_f(\x^m) <\frac{-v_f(\x^0)}{2}.$$

Let's consider the weighted sum of all the above inequalities for $j=0, \ldots m-1$ with weights $\lambda^j$'s to get:
$$\left\langle \sum_{j=0}^{m-1}\lambda^j\g^j, \x^*-\x^0\right\rangle - \sum_{j=0}^{m-1}\lambda^j \left\langle  \g^j, \x^j-\x^0\right\rangle < \frac{-v_f(\x^0)}{2}\left( \sum_{j=0}^{m-1}\lambda^j\right),$$
which can be rearranged to the following form, 
$$\left\langle  \sum_{j=0}^{m-1}\lambda^j\g^j, \x^*-\x^0\right\rangle <  - \frac{v_f(\x^0)}{2}\left( \sum_{j=0}^{m-1}\lambda^j\right) + \sum_{j=0}^{m-1}\lambda^j \left\langle  \g^j,\x^j-\x_0\right\rangle.$$

Equivalently we can rewrite the above inequality as:
\begin{align*}\left\langle  \sum_{j=0}^{m-1}\lambda^j\g^j, \x^*-\x^0\right\rangle < &- \frac{v_f(\x^0)}{4}\left( \sum_{j=0}^{m-1}\lambda^j\right) \\ &+ \left( \frac{f(\x^*)-f(\x^0)}{4}\left( \sum_{j=0}^{m-1}\lambda^j\right) +\sum_{j=0}^{m-1}\lambda^j \left\langle  \g^j, \x^j-\x^0\right\rangle \right).\end{align*}

Using inequality \eqref{eq1} along with the facts that $f(\x^*)\le f(\x^j)$ and $\lambda^j \ge 0$ for all $j$, we get:
\begin{equation} \label{eq4} \left\langle  \sum_{j=0}^{m-1}\lambda^j\g^j, \x^*-\x^0\right\rangle <  - \frac{v_f(\x^0)}{4}\left( \sum_{j=0}^{m-1}\lambda^j\right). \end{equation}

By the Cauchy-Schwarz inequality we have
$$-\left\| \sum_{j=0}^{m-1}\lambda^j\g^j\right\| \left\| \x^*-\x^0\right\| \le  \left\langle  \sum_{j=0}^{m-1}\lambda^j\g^j, \x^*-\x^0\right\rangle <  -\frac{v_f(\x^0)}{4}\left( \sum_{j=0}^{m-1}\lambda^j\right), $$
hence
\begin{equation}\label{eq5}
\left\| \sum_{j=0}^{m-1}\lambda^j\g^j\right\| \left\| \x^*-\x^0\right\| >  \frac{v_f(\x^0)}{4}\left( \sum_{j=0}^{m-1}\lambda^j\right).
\end{equation}

By property \eqref{prop3} we have
\begin{equation}\label{eq6} \left\| \x^*-\x^0\right\| \le \sqrt{\frac{2v_f(\x^0)}{l}}. \end{equation} 

Furthermore, by inequalities \eqref{eq2} and \eqref{eq3} we get:
\begin{equation}\label{eq7}\left\| \sum_{j=0}^{m-1}\lambda^j\g^j\right\| \le \rho \sqrt{ \sum_{j=0}^{m-1}\left(\lambda^j\right)^2\| \g^j\|^2}\le  \rho\sqrt{v_f(\x^0)}.\end{equation}

Replacing inequalities \eqref{eq6} and \eqref{eq7}  in inequality \eqref{eq5}, we get 
\begin{equation}\label{eq8}
\rho\sqrt{v_f(\x^0)}\sqrt{\frac{2v_f(\x^0)}{l}}>  \frac{v_f(\x^0)}{4}\left( \sum_{j=0}^{m-1}\lambda^j\right).  \end{equation}

Notice that by definition of $\lambda$ and property \eqref{prop1}, $\lambda^j \ge \sqrt{\frac{1}{2L}}$ for all $j$, so 
$$\sum_{j=0}^{m-1}\lambda^j \ge \sqrt{\frac{1}{2L}}\ m.$$

Using this fact in inequality \eqref{eq8}, we get
$$\rho\sqrt{v_f(\x^0)}\sqrt{\frac{2v_f(\x^0)}{l}}>  \frac{v_f(\x^0)}{4}\left(\sqrt{ \frac{1}{2L}} \ m \right),$$ 
therefore 
\begin{equation}\label{eq9}
m< 8\rho \sqrt{\frac{L}{l}},
\end{equation}
which contradicts our assumption on the value of $m$.
\qquad
\end{proof}

Lemma \ref{lem1} shows that under conditions \eqref{eq1} and
\eqref{eq2}, the residual of the function is divided in half every
$m=O(\sqrt{\frac{L}{l}})$ iterations. For the next sequence of $m$
iterations, a further reduction of $\frac{1}{2}$ is achieved provided
\eqref{eq1} and \eqref{eq2} hold, with $\x^m$ substituted in place of
$\x^0$. Hence by letting $\x^m$ be the new $\x^0$ and repeating the
same algorithm, we can find the $\epsilon$-optimal solution in
$\lt\lceil \log_2 \frac{1}{\epsilon} \rt\rceil \lt\lceil 8\rho
\sqrt{\frac{L}{l}} \rt\rceil$ iterations. Nemirovsky and Yudin's
algorithm follows this outline: it
is designed to ensure that \eref{eq1} and \eref{eq2} hold
on every iteration, and it restarts every $m$ iterations.
For ordinary nonlinear CG, however, there is no
assurance that these inequalities will hold, and, furthermore, $m$ is not
known.  These issues motivate our
detection and correction steps.

\section{Detecting loss of independence}
\label{sec:detection}

As mentioned in the previous section, we take ``loss of independence'' to mean
failure of \eref{eq1} or \eref{eq2}.  In this section we describe a method
to detect the failure of these inequalities.  

Before turning to \eref{eq1} and \eref{eq2}, 
we note that the lemma can also fail if condition
(a), namely, the requirement that
$f(\x^{j+1})\le f(\x^j)-\frac{1}{2L}\| \g^j\|^2$,
fails to hold.  If we had prior knowledge of $L$, then this condition
would be trivial to check since nonlinear CG already computes $\g^j$ on every
iteration.  Without prior knowledge of $L$, we can still in principle check this
condition by carrying out a Wolfe line-search \cite{nocedalwrightbook}
in the direction $-\g^j$ on 
every iteration.  It is known that, up to a constant factor depending on the
parameters $\beta,\sigma$ used in the line-search, the reduction guaranteed
is at least as good as $\| \g^j\|/(2L)$.  However, it is quite expensive
to carry out a line search in the steepest descent direction on every iteration
in addition to the line search already required for the CG direction.
Our computational experiments (not reported here) indicate
that it is not necessary because there is little improvement in the behavior
of the method.  Therefore,
for the rest of this paper, we will simply assume that (a) holds.

We next turn to \eref{eq1} and \eref{eq2}.  It is apparent from their form
that they can be checked efficiently by keeping running totals of all the
summations appearing in them.  This is how we have implemented them.  The extra
cost for tracking these summations is very low compared to the existing cost
of evaluating $f$ and $\nabla f$ in an ordinary CG iteration.
As mentioned at the end of the previous section,
every $m$ iterations, we need to replace
$\x^0$ by $\x^{im}$ for integer values of $i$ in \eref{eq1} and \eref{eq2}
in order to obtain the theoretical convergence result.

This replacement of $\x^0$ by $\x^{im}$ is a sticking
point because $m$ is not known in advance:
the algorithm does not have prior knowledge of $L$ or $l$.  
Furthermore, for some classes of strongly convex functions such as
log-barrier functions, the effective value of $L/l$ may decrease
as the optimizer is approached.
We address this
difficulty as follows.  Although $L/l$ is not known, we can be certain that
there is some nonnegative integer $p$ such that $L/l\in [2^p, 2^{p+1}]$.
Therefore, we
maintain $p_{\max}$ separate sets of running totals, where 
$p_{\max}=\lceil\log_2 j\rceil$, where
$j$ is the current iteration counter.  In other words, for each
$p\in\{0,\ldots,p_{\max}\}$, we maintain a current value of the
summation $\sum_{j'=m(j,p)^j}\lambda^{j'}$ and so on for all the summations appearing
in \eref{eq1} and \eref{eq2}.  Here $m(j,p)$ denotes the largest multiple of
$2^p$ less than or equal to the current iteration counter $j$.
Once $j$ reaches
the next multiple of $2^p$, we can check the inequalities for this particular
value of $p$.  This additional work for updating the $p_{\max}$
running totals and checking the inequalities is still
insignificant compared to the work of evaluating the gradient and carrying
out a line-search; it adds an additional $O(\log j)$ arithmetic operations
to the $j$th iterate.

In fact, there is little harm in omitting the check on the conditions
for very small values of $p$ since the lemma will still guarantee convergence,
albeit slightly more slowly, if we catch those corrections for larger
values.  For this reason, the conditions are actually tested only
for $p\ge p_l$ in our implementation, where we have taken $p_l=4$.

This concludes our description of the detection procedure.  If the failure
of these inequalities is repeatedly detected, this is an indicator that
loss of independence has occurred.  

\section{Correcting the loss of independence}
\label{sec:correction}

It is already useful to be able to detect loss of independence, since this
is a sign that conjugate gradient may not be working.  One possibility when
loss of independence is detected is to simply restart.  As mentioned in
the introduction, restarting is the conventional solution to loss
of independence in CG.

We have instead adopted a more comprehensive solution, namely, we
propose a correction procedure to ensure that \eref{eq1} and
\eref{eq2} are guaranteed to hold.  The correction procedure is
similar to the subspace optimization proposed by Nemirovsky and Yudin.
A consequence of our correction procedure is that we are assured that
their theoretical complexity bound of $O(|\ln \epsilon|\sqrt{L/l})$
iterations
holds for nonlinear CG if our correction procedure is instituted.
Furthermore, we have an advantage over the Nemirovsky-Yudin algorithm
that prior knowledge of $L/l$ is not required.  On the other hand, we
have a disadvantage that the dimension of the subspace could be larger
than 2 (their dimension), and hence our iterations can be more expensive.

Let us refer to the sequence of iterates between two consecutive multiples of $2^p$
as
a ``block'' of iterates; in other words, for any $p$, the sequence of
iterates $\x^0, \x^1,\ldots, \x^{2^p-1}$ is the first block of size
$2^p$, $\x^{2^p}, \x^{2^p+1}, \ldots, \x^{2(2^p)-1}$ is the second
block of size $2^p$, and so on.  At the end of each block we check
inequalities \eqref{eq1} and \eqref{eq2}. If they are satisfied and
$2^p\ge \lt\lceil8\rho\sqrt{\frac{L}{l}}\rt\rceil$, then by Lemma
\ref{lem1} we know that the residual of the function is divided in
half; however if any of these inequalities fails, then 
we need to take a ``correction step" for the next
block of iterates. The correction step involves computing the next
block of iterates in a way that satisfaction of inequalities
\eqref{eq1} and \eqref{eq2} is guaranteed at the end of this
block. Then the correction step is omitted in the subsequent blocks
until the inequalities are violated again.

Suppose at least one of the inequalities \eqref{eq1} and \eqref{eq2}
is violated for $k$th block of $p$; i.e. for the block of iterates
$\x^{r_p}, \ldots, \x^{r_p+2^p-1}$ where $r_p=(k-1)2^p$. Then for the
next block we search for the new iterate $\x^{j+1}$ on the space of
$\x^j+\text{\emph {Sp}} \left\{\g^j, \d^j, \q^j_p, \x^j-\x^{r_p}
\right\}$ where $\q^j_p=\sum_{i=r_p}^j \lambda^i\g^i$.
Notice that this space includes the conjugate gradient search direction
(all variants) because it is a linear combination of $\g^j$
and $\d^j$.

Finding the new iterate $\x^{j+1}$ through a search on the space that
in addition to $\g^j$ and $\d^j$ includes $\q^j_p$ and $\x^j-\x^{r_p}$
is what we referred to as ``correction step". Notice that for each $p$
with the violated constraints we increase the dimension of the search
space by 2. However, the dimension of the search space never exceeds
$2+2\lceil p_{\max}-p_l+1\rceil$, which happens to be the case when
the inequalities are violated for all possible values of $p$.  (Recall
that we check the inequalities for $p=p_l,p_{l+1},\cdots,p_{\max}$,
on iteration $j$, where $p_l=4$ in our implementation and
where $p_{\max}=\lceil \log_2 j\rceil$.)

It is quite easy to see that inequalities \eqref{eq1} and \eqref{eq2}
are satisfied for the $(k+1)$st block of $p$ when we take the
correction step throughout it. By KKT condition, we have $\left\langle
\g^{j}, \x^{j}-\x^{r_p}\right\rangle =0$ for all $ j$ in this
block. Using this, along with the fact that $f(\x^j)<f(\x^{r_p})$ and
non-negativity of $\lambda^j$ for all $j$, we derive
\eqref{eq1}. Similarly one can argue that by KKT $\left \langle \g^j,
\q^{j-1}_p\right\rangle=0$ for all $j$, hence
$$ \left \| \sum_{i=r_p}^{r_p+2^p-1} \lambda^i\g^i\right\| = 
\sqrt{ \sum_{i=r_p}^{r_p+2^p-1} \left(\lambda^i\right)^2 \| \g^i\|^2},$$
which means inequality \eqref{eq2} is satisfied. 
(Notice that this equation holds provided $\g^j$ is orthogonal to
the previous running total of weighted gradients; 
it is not necessary for $\g^j$ to be orthogonal
to each previous gradient.)

After finding the iterates of one block through a correction step, the algorithm switches back to taking a regular step until the next failure of the inequalities.

We have implemented two procedures for subspace optimization: Newton's
method and the ellipsoid method.  We used Newton's method unless it
fails to rapidly converge to the optimum. Note that the assumption of
strong convexity is not a sufficient condition for convergence of
Newton's method, but it succeeds in many cases nonetheless.
In the case of failure of
Newton's algorithm, the ellipsoid method carries out the task of
solving the optimization problem. In other words, we impose an upper
bound to the number of iterations that Newton's method may take, and
if it fails to converge within the given number of iterations, the
algorithm switches to the ellipsoid method for solving the subspace
problem. 

Recall that at $(j+1)$st iterate, we search for $\x^{j+1}$ in the
space of vectors $\x=\x^j + \alpha \g^j + \beta \d^j + Q\a  +  R\b$,
where $Q\in \mathbb R^{n\times |S|}$ is the matrix formed by columns
$\q_p^j$ for all $p\in S$; $R$ is the matrix of the same dimension
with columns $\x^j-\x^{r_p}$ for all $p\in S$; $\alpha, \beta \in
\mathbb R$, and $\a, \b\in \mathbb R^{|S|}$ are coefficients that we
want to find.  Here, $S\subset\{p_l,\ldots,p_{\max}\}$ denotes
the set of indices for which correction is required.

Let $\y$ denote the variable of the subspace
optimization problem, i.e., $\y=\left[ \alpha, \beta, \a^t, \b^t
  \right]^t$; in addition let $B= \left[\g^j, \ \d^j, \ Q, \ R
  \right]$ and $K=2+2\left|S\right|$. We can now state the formal
presentation of the subspace optimization problem,

\begin{equation}\label{subprob}
\min_{\y\in \mathbb R^{K}} f\left(\x^j+B\y\right)
\end{equation}

As mentioned above, we first attempt to solve problem \eqref{subprob} with Newton's method. Letting $\tilde f(\y)=f(\x^j+B\y)$ and using chain rule we get the following formulas for the gradient and Hessian of each Newton's iteration, 
\begin{equation}\label{gradnewton}
\nabla \tilde f(\y)= B^t\nabla f(\x)
\end{equation}
\begin{equation}\label{hessnewton}
\nabla^2 \tilde f(\y)= B^t\nabla^2 f(\x) B
\end{equation}

Notice that some second order information of the function comes into
play in equation \eqref{hessnewton}. We compute $\nabla
f(\x)$ and $\nabla^2 f(\x)$ directly when $f(\x)$ is simple
enough. For more complicated functions we use automatic
differentiation (AD) in backward mode to compute $\nabla f(\x)$ and
$\nabla^2 f(\x) B$. Let $B^{(k)}$ denote $k$th column of matrix
$B$. Backward AD enables us to keep the computational cost of $\nabla
f(\x)$ within a constant factor of the objective function evaluation
cost, and the cost of computing $\nabla^2 f(\x)B^{(k)}$ within a
constant factor of the computational cost of gradient evaluation multiplied
by the number of columns of $B$. The
storage space required in backward AD, however, is more than the
required storage in forward AD; and in worst case it can be
proportional to the number of operations required for computing
$f(\x)$. We did not use an AD tool but rather derived second derivative
routines by hand.  Details on our test
problems are presented in Section~\ref{sec:comput}. For more information
on AD, one may refer to \cite{nocedalwrightbook}.

In addition to the storage required by AD, we need to store $\x^j$,
and matrix $B$; we also need to update and store $\x^{r_p}$,
$\sum_{i=r_p}^{j} \lambda^i$, $\sum_{i=r_p}^{j} \lambda^i \left\langle
\g^i,\x^i-\x^{r_p}\right\rangle$, $\sum_{i=r_p}^{j} \lambda^i\g^i$,
$\sum_{i=r_p}^{j} (\lambda^i)^2\left \| \g^i \right \|^2$ for all
$p\in \left\{ p_l, \ldots, p_{\max}\right\}$. The
required storage space for the above elements is in $O(n\lceil
\log_2j\rceil\})$.

The subspace optimization with either Newton's method or the ellipsoid
method needs a termination test.  For this purpose, we again
rely on \eqref{eq1} and \eqref{eq2}.  Although the lemma requires
these inequalities to be checked only at an iteration
at the end of a block, it
is also possible to check them on intervening iterations.  We use
these inequalities to terminate the search for a subspace solution.
Note that at an exact solution
to the subspace problem, the inequalities are sure to hold because
of the KKT conditions of the subspace problem, as already mentioned.

We can now present the algorithm in its entirety. We call it
CGSO for ``conjugate gradient with subspace optimization.''
In this procedure, $S$ is a subset of $\{p_l,\ldots,p_{\max}\}$ and
denotes the set of values of $p$ for which correcting is currently
active. To save space, we use the Python tabbing convention that 
the end of a code-block is denoted by a retraction of the indent-level.

\begin{alg}\label{alg1}
\begin{tabbing}
\\
+\= +++ \= +++\=+++\=+++ \=+++\=\kill
\> SUBROUTINE: \verb+verify_step+$(\x^j,\s^j)$ \\
\>\> {\bf for} each $p \in S$ \\
\> \>\>   {\bf if} \eref{eq1} or \eref{eq2} fail with $\x^{m(j,p)}$ 
substituted for $\x^0$ and $\x^j+\s^j$ substituted 
for $\x^{m-1}$\\
\>\>\>\> {\bf return} {\em False}; \\
\>\> {\bf return} {\em True};\\
\> MAIN PROCEDURE: \verb+CGSO+$(\x^0)$ \\
\>  $S=\emptyset$\\ 
\> {\bf for} $j=1,2,\ldots$\\
\>\>  $\d^j = -\g^j+\beta^j\d^{j-1}$;  \\
\>\>\>\>Remark: this is the ordinary nonlinear CG direction. \\
\>\>\>\> Remark: take $\beta^j=0$ if either $j=1$ or $\d^{j-1}$ was discarded.\\
\>\>  $\alpha^j= \verb+Wolfe_line_search+(f,\x^j,\d^j);$ \\
\>\>  stepfound = {\em False}; \\
\>\>  {\bf if} \verb+verify_step+$(\x^j,\alpha^j\d^j)$ \\
\>\>\>  stepfound = {\em True}; \\
\>\>\>  $\s^j=\alpha^j\d^j$; \\
\>\> {\bf else} \\
\>\>\> discard $\d^j$;\\
\>\> {\bf if not} stepfound \\
\>\>\>  Apply Newton's method to solve \eref{subprob}.\\
\>\>\>  Terminate if either \verb+verify_step+$(\x^j,B\y^l)$
 or iteration-max is attained.\\
\>\>\> {\bf if} \verb+verify_step+$(\x^j,B\y^l)$ \\
\>\>\>\>  stepfound = {\em True}; \\
\>\>\>\>  $\s^j=B\y^l;$ \\
\>\> {\bf if not} stepfound \\
\>\>\>  Apply the ellipsoid method to solve \eref{subprob}.\\
\>\>\>  Terminate when \verb+verify_step+$(\x^j,B\y^l)$. \\
\>\>\>  stepfound = {\em True}; \\
\>\>\>  $\s^j=B\y^l;$ \\
\>\> $\x^{j+1} = \x^{j}+\s^{j};$\\
\>\>  {\bf for} $p=p_l,\ldots, \lceil \log_2j\rceil$\\  
\>\>\> {\bf if} $j+1 = k_p2^p$ for some integer $k_p$\\
\>\>\>\> {\bf if} $p\in S$\\
\>\>\>\>\> $S=S\setminus \{p\}$\\
\>\>\>\> {\bf elseif not} \verb+verify_step+$(\x^{j-1},\s^{j-1})$; \\
\>\>\>\>\>  $S=S\cup \{p\}$
\end{tabbing}

 \end{alg}

As mentioned earlier, because the above algorithm enforces
\eref{eq1} and \eref{eq2} for every value of $p$ and for at
least every other block, we get the optimal convergence bound.
 
 \begin{theorem}\label{thm1}
Suppose $m\ge \left \lceil 8\rho \sqrt {\frac{L}{l}}\right\rceil$, and $\x^j$ is a sequence generated by Algorithm $\ref{alg1}$ for solving problem \eqref{probdef}.
Then for any integer $n\ge 0$, $v_f(\x^{(n+4)m})\le \frac{1}{2}v_f(\x^{nm})$.
\end{theorem}

\begin{proof}
Let $\bar p$ be the integer for which $2^{\bar p-1}\le m \le 2^{\bar p}$; and let $s_{\bar p}$ stand for $2^{\bar p}$. Using algorithm \ref{alg1}, we are guaranteed that for at least one of any two consecutive blocks of size $s_{\bar p}$ inequalities \eqref{eq1} and \eqref{eq2} are satisfied. The size of this block is $s_{\bar p} \ge m \ge \left \lceil 8\rho \sqrt {\frac{L}{l}}\right\rceil$ and hence by Lemma \ref{lem1} we have
\begin{equation}\label{thm1eq1}v_f(\x^{nm +2s_{\bar p}})\le \frac{1}{2}v_f(\x^{nm}).\end{equation}

Since $2s_{\bar p}\le 4m$, so $f(\x^{nm+4m})\le f(\x^{nm+2s_{\bar p}})$; hence 
\begin{equation}\label{thm1eq2}v_f(\x^{nm+4m})\le v_f(\x^{nm+2s_{\bar p}}).\end{equation}

\eqref{thm1eq1} and \eqref{thm1eq2} gives us the result we wanted to show.
\qquad
\end{proof}

\section{Remarks on computational divided differences}
\label{sec:divdiff}
In a line-search for conjugate gradient, it is necessary to accurately
evaluate quantities of the form $f(\x+\alpha\d)-f(\x)$.  A similar
quantity arises in the ratio test for the trust-region method
\cite{nocedalwrightbook}.  It is well known to implementors of such
methods that these divided differences are problematic near the
root because of cancellation error between the two terms.  A brief
discussion of this issue appears in Hager and Zhang \cite{hagerzhangdescentcg}.
Failure to compute these quantities accurately can lead either to
premature termination of an algorithm or to infinite loops.

A solution to this problem, perhaps not as widely known in the optimization
literature as it should
be, is ``computational divided differences'' by Rall and Reps \cite{RepsRall}.
The idea is
to transform a source-code program for computing $f$ into another 
source-code program for accurately computing divided differences of
$f$.  The technique is somewhat reminiscent of automatic differentiation.

To give a concrete example, consider the log-barrier function that will
be used in Section~\ref{sec:comput} as a test case, which is written as
$f(\x)=\sum_{i=1}^m\log(\a_i^T\x-b_i)$, where each $\a_i$ is given
vector in $\R^n$ and each $b_i$ is a given scalar.  
This function is
defined on the open polyhedron given by $A\x>\b$ and strongly convex
on this polyhedron provided that the polyhedron is bounded.
Suppose $\x$ is our current iterate and $\bdelta$ is a small step.
We have the following
derivation:
\begin{eqnarray*}
f(\x+\bdelta)-f(\x) & =
& \sum_{i=1}^m\log(\a_i^T(\x+\bdelta)-b_i)- \sum_{i=1}^m\log(\a_i^T\x-b_i) \\
& = & 
\sum_{i=1}^m\log\left(1 + \frac{\a_i^T\bdelta}{\a_i^T\x-b_i}\right).
\end{eqnarray*}
Thus, to evaluate this divided difference accurately, one needs a function
to compute $\log(1+a)$ accurately when $|a|$ is small.  One can develop
a method for this computation using calculus.  That
effort is, however,  unnecessary since Matlab and C++
both contain the built-in library function \verb+log1p+ for exactly this purpose.

We have used computational divided differences for all of our testing.
(We hand-coded the accurate
divided differences rather than using a source-to-source
translation tool; we are not sure if such a tool exists.)
In addition to the line-search,
our method uses computational divided differences 
for the evaluation of the left-hand side of \eref{eq1}.
Without them, all the methods would be less reliable and the
test results harder to interpret.  Indeed, we believe that computational
divided differences deserve to be used much more widely in general
nonlinear optimization than they are currently.
See also the unpublished note by the second author \cite{Vavasis:divdif}
for some comments on their use in optimization.

Because of our reliance on this technique, however, it is not possible
to directly compare our results in the next section to well known packages
like CG-DESCENT, which do not use computational divided differences.  For
this reason, we compare only our own implementations against each other.

\section{Computational experiments}\label{sec:comput}

We have tested the correction method on four classes of problems, three convex
and the fourth nonconvex.  Our setup was as follows.  We tried three different
variants of conjugate gradient, namely FR, PR+, and HZ.  Here, PR+ denotes
the Polak-Ribi\`ere method in which the parameter $\beta$ is replaced by 0
in the case that it becomes negative (thus forcing a restart), which
is a recommended modification (see \cite{nocedalwrightbook}).  HZ refers
to the CG-descent direction of Hager and Zhang \cite{hagerzhangdescentcg}.

Most of our test cases are small.  This allowed us to perform more experiments
in a reasonable amount of time.  As mentioned earlier, the behavior of conjugate
gradient is governed much more by conditioning of the problem than problem size.
However, to illustrate that the method is also suitable for large problems, we
have included two somewhat larger test cases.

The results of our experiments can be summarized as follows.  For uncorrected
methods, the HZ direction is usually the best while the FR method is usually
the worst, and sometimes FR is much worse.  For corrected methods, all three
directions perform about equally.  The corrected methods are typically slower
than the uncorrected HZ method for well-conditioned problems.  For ill-conditioned problems,
however, the corrected method is sometimes much better than HZ (as well as the
other two methods).  Note that no forced
restarts have been implemented.  However, there are still restarts in some
cases.  As noted above, in our correction procedure, when a conjugate
gradient search direction is discarded, the following step is, at least initially,
the 
steepest descent direction.  
Also as noted above, the PR+ method will sometimes restart
automatically if it computes a negative $\beta$.

Before presenting the results, we need to comment on how the running
time was measured.  We measure time in ``units'', where
we count as one unit an evaluation of a function
or gradient or function/gradient pair (at the same point).  In the
line-search procedure, gradients are evaluated several times, so each
outer iteration costs several units.  (Our line search is based on
simple bisection and the Wolfe conditions.)  
We count the evaluation of $\nabla^2f(\x)\y$, needed
for Newton's method, 
as two units.  Here, $\x$ and $\y$ are arbitrary vectors.
In fact, this is a
simplification since the cost varies for different functions.  For
example, in the case of a quadratic function, the cost of
$\nabla^2f(\x)\y$ is actually the same as the cost of $\nabla f(\x)$
(one matrix-vector multiplication).  The main theorem of backward-mode
automatic differentiation states that the evaluation of
$\nabla^2f(\x)\y$ should never cost more than 5 units.  (None of our
examples reach this upper bound of 5.)  Finally, one iteration of the ellipsoid
method also counts as one unit since it involves one gradient evaluation.

We now present the results in more detail.  The first test
function is a simple quadratic, $f(\x)=\x^TA\x+\b^T\x$ for a positive
definite matrix $A$.  Note that none of the methods reduce to
linear CG in this case because we did not implement an exact line
search.
Therefore, there is no prior guarantee that independence of search
directions is maintained.
On the other hand, because the problem is quadratic, the Newton
method on the subspace converges in a single iteration and the ellipsoid
method is never used.
In two cases we formed $A$ by choosing 1000 geometrically
spaced eigenvalues in a predetermined interval and then multiplying on the
left and right by a random $1000\times 1000$ orthogonal matrix.  In this way, the
condition number of $A$ is determined exactly.  In the third case we
formed $A$ as the assembled stiffness matrix of a finite-element discretization
of Poisson's equation on the unit disk with a relatively uniform and well-behaved 
mesh. This problem has moderate ill-conditioning, but the matrix was too large
to exactly measure its condition.
The results of
these experiment 
are shown in Table~\ref{tab:quadres}.

\begin{table}
\begin{center}
\caption{Number of units of computation for convex quadratic functions;
the first two lines are smaller problems ($n=1000$);
the last line is a larger finite-element problem,  $n=197,136$.  
An asterisk indicates a computation terminated due to an iteration limit.}
\label{tab:quadres}
\begin{tabular}{lrrrrrr}
\hline 
  & \multicolumn{3}{c}{Uncorrected} & \multicolumn{3}{c}{Corrected}\\
  & HZ & FR & PR+ & HZ & FR & PR+ \\
\hline
$\mbox{cond}(A)=10^5$ & 
38,483 & 98,442 & 73,756 & 87,894 & 85,930 & 88,280 \\
$\mbox{cond}(A)=10^8$ & 
*5,552,754 & 8,557,387 & *27,669,107 & 2,407,560 & 2,181,492 & 2,517,924\\
cond unknown &
149,543 & 66,373 & 400,112 & 115,698 & 110,230 & 86,200\\
\hline
\end{tabular}
\end{center}
\end{table}

The next class of experiments is with log-barrier functions, that is,
functions of the form $f(\x)=\mu\sum_{i=1}^m\log(\a_i^T\x-b_i)+\c^T\x$.  In these
experiments we generated $A$ randomly with known condition number for two
smaller cases, and we took $A$ to be the node-arc incidence matrix of an undirected
graph (hence two copies of each edge, one for each direction) for a larger
test case.  This matrix $A$ is relatively well conditioned.  However, we
can make the problem more ill-conditioned by decreasing $\mu$ (thus pushing
the solution closer to the boundary of the feasible region).
The graph in question
came from a DIMACS challenge problem. The results are in Table~\ref{tab:logbar}.

\begin{table}
\begin{center}
\caption{Number of units of computation for log-barrier functions.
The first three lines are smaller problems ($A\in\R^{400\times 100}$);
the last line is a larger DIMACS graph problem ($A\in\R^{91,756\times15,605}$).
In the third line, the condition number of $A$ was slightly worse.
An asterisk indicates a computation terminated due to an iteration limit.}\label{tab:logbar}
\begin{tabular}{lrrrrrr}
\hline 
  & \multicolumn{3}{c}{Uncorrected} & \multicolumn{3}{c}{Corrected}\\
  & HZ & FR & PR+ & HZ & FR & PR+ \\
\hline
$\mu=.4$ & 
292,012 & 1,496,650 & 963,968 & 461,036 & 394,382 & 420,808 \\
$\mu=.1$ &
593,190 & 3,034,394 & 2,059,235 & 1,568,258 & 1,477,938 & 1,463,452 \\
$\mu=.1$ &
*55,190,257 & *58,728,472 & *55,665,606 & 13,349,163 & 15,235,567 & 14,813,917 \\ 
$\mu=100$ &
1,298,292 & *6,633,403 & 2,297,573 &  762,649 & 654,900 & 668,127 \\
\hline
\end{tabular}
\end{center}
\end{table}

The third test case consists of smoothed versions of the LASSO problem.
The unsmoothed version of this problem has an objective function of the
form $\Vert A\x-\b\Vert^2+\lambda\Vert\x\Vert_1$, where $A$ has fewer rows
than columns.  In the smoothed version we approximate the function $|x|$ by
$(x^2+\delta)^{1/2}$ which is convex (strongly convex on bounded intervals) and smooth.
We did not try a large instance of this problem because typically $A$ is taken
to be a dense matrix, so a large problem would require too much computation
time.  The results are in Table~\ref{tab:lasso}.

\begin{table}
\begin{center}
\caption{Number of units of computation for regularized LASSO functions. For
each case, $A\in\R^{100\times 400}$.  For both rows, the regularization parameter
$\delta$ is $5\cdot 10^{-4}$.}
\label{tab:lasso}
\begin{tabular}{lrrrrrr}
\hline 
  & \multicolumn{3}{c}{Uncorrected} & \multicolumn{3}{c}{Corrected}\\
  & HZ & FR & PR+ & HZ & FR & PR+ \\
\hline
$\lambda=10^{-3}$, $\mbox{cond}(A)=10^5$
&
51,514 & 263,202 & 116,740 & 97,781 & 99,072 & 91,763 \\
$\lambda=10^{-4}$, $\mbox{cond}(A)=10^6$
&
986,314 & 5,049,449 & 3,397,063 & 810,887 & 926,389 & 879,384 \\
$\lambda=10^{-4}$, $\mbox{cond}(A)=10^6$
&
46,206,176 & 56,618,846 & 55,827,526 & 15,523,751 & 19,667,056 & 14,350,066\\
\hline
\end{tabular}
\end{center}
\end{table}

The final test case is the nonconvex distance geometry problem.  In this problem,
there is a sequence of $n$ points $(\x_1,\ldots,\x_n)$ each
in $\R^d$ whose coordinates are mostly unknown.  However,
many pairs of interpoint distances are given.  The problem is to find
the positions of the points.  This can be posed as a nonlinear least squares
problem of minimizing 
$\displaystyle \sum_{(i,j)\in E}\left(d^2_{i,j}-\Vert\x_i-\x_j\Vert^2\right)^2$
where $E$ is a list of the pairs $(i,j)$ whose distances are known, $d_{ij}$ is the
known distance, and the $\x_i$'s are unknown (except for a few, called `anchors',
which make the problem well posed).

Because of the nonconvexity, it is possible for different algorithms to converge
to different local optimizers; such a result would naturally make the running
time estimates difficult to interpret.  In order to prevent this inconsistency,
the data was constructed so that there is an exact solution (i.e., the
nonlinear least squares instance has a solution with zero residual), and then
all the methods were initialized at a point close to that solution.  With this device, 
we were able to ensure convergence to the same solution.  The coordinates of the
known solution were taken as random points in the plane, and a random subset
of possible edges was used in the objective function.

A second issue with nonconvexity is that the ellipsoid method is no longer
valid for solving the subspace problem.  Therefore, our two methods for solving
the subspace problem in this case were Newton, and, if it fails, a trust-region
method \cite{nocedalwrightbook}.  However, it turned out that the trust-region
method was never invoked, most likely because we started sufficiently close 
to the root.  We can control the conditioning of the problem by stretching
the random data points along one axis ($x$ or $y$).  The results of a well-conditioned
and ill-conditioned problem are in Table~\ref{tab:distgeo}.

\begin{table}
\begin{center}
\caption{Number of units of computation for distance geometry functions. In
each case the number of unknowns was 400 while the number of distances was 600.}
\label{tab:distgeo}
\begin{tabular}{lrrrrrr}
\hline 
  & \multicolumn{3}{c}{Uncorrected} & \multicolumn{3}{c}{Corrected}\\
  & HZ & FR & PR+ & HZ & FR & PR+ \\
\hline
stretch=1
&
29,829 & 61,148 & 56,453 & 42,311 &  38,501 & 43,123 \\
stretch=5
&
328,436 & 672,881   & 974,868  & 87,416 & 93,771 & 102,012 \\
\hline
\end{tabular}
\end{center}
\end{table}

\section{Conclusion}\label{conclusion}
We have presented an analysis of loss of independence in conjugate gradient
search directions.  The analysis is derived for strongly
convex functions and is based on work by Nemirovsky and Yudin.  
The analysis suggests a correction method involving subspace optimization
on many iterations.  The dimension of the subspace is at least $4$ and is
bounded above in terms of the log of the current iteration counter.

The correction method, though expensive, appears to lead to the fastest
solution in the case of ill-conditioned instances.  When the correction method
is used, there is seemingly little difference between the three variants
of conjugate gradient, FR, PR+ and HZ that we tested.
Although the method was based on theory developed for the strongly convex
case, convexity is not inherent in the formulas themselves and so it straightforward
to extend the correction to the nonconvex case.
Finally, this work advocated for greater use of computational divided
differences in the optimization community.

This work raises several questions.  On the theoretical side, it would
be interesting to have a method that can be classified as nonlinear conjugate
gradient (i.e., reduces to linear CG when applied to a quadratic
function) but achieves the optimal complexity bound of
$O(|\ln(\epsilon)|\sqrt{L/l})$ function/gradient evaluations in the
general case of strongly convex functions.  Although our CGSO method achieves
this iteration bound, it does not achieve
the same bound for function/gradient evaluations
because we do not have a constant upper bound on the number
of inner iterations needed for subspace optimization.  It would also
be interesting to have some kind of analysis, even a weak result,
of the correction method for nonconvex problems.

On the practical side, it would be interesting to understand why the three
nonlinear CG methods, which often exhibit widely varying behavior, seem
to become nearly indistinguishable once our correction method is applied. 
\newpage
\bibliographystyle{plain}	
\bibliography{CGref}

\end{document}